\documentclass[11pt,a4paper]{article}
\usepackage{bbm}
\usepackage{epsf,epsfig,amsfonts,amsgen,amsmath,amstext,amsbsy,amsopn,amsthm,lineno}
\usepackage{color}
\usepackage{ebezier,eepic}
\newtheorem{theorem}{Theorem}

\newtheorem{lemma}{Lemma}
\newtheorem{false statement}{False statement}

\theoremstyle{definition}

\newtheorem{claim}{Claim}
\newtheorem{subclaim}{Claim}[claim]

\newtheorem{remark}{Remark}

\newtheorem{case}{Case}

\newcommand{\de}{{\rm def}}

\begin{document}

\title{\bf\Large Heavy subgraph conditions\\
for longest cycles to be heavy in graphs}

\date{July 13, 2011 }

\author{Binlong Li and Shenggui Zhang\thanks{Corresponding author.
E-mail address: sgzhang@nwpu.edu.cn (S. Zhang).}\\[2mm]
\small Department of Applied Mathematics,
\small Northwestern Polytechnical University,\\
\small Xi'an, Shaanxi 710072, P.R.~China} \maketitle

\begin{abstract}
Let $G$ be a graph on $n$ vertices. A vertex of $G$ with degree at
least $n/2$ is called a heavy vertex, and a cycle of $G$ which
contains all the heavy vertices of $G$ is called a heavy cycle. In
this paper, we characterize the graphs which contain no heavy
cycles. For a given graph $H$, we say that $G$ is $H$-\emph{heavy}
if every induced subgraph of $G$ isomorphic to $H$ contains two
nonadjacent vertices with degree sum at least $n$. We find all the
connected graphs $S$ such that a 2-connected graph $G$ being
$S$-heavy implies any longest cycle of $G$ is a heavy cycle.

\smallskip
\end{abstract}

\section{Introduction}

We use Bondy and Murty \cite{Bondy_Murty} for terminology and
notation not defined here and consider finite simple graphs only.

Let $G$ be a graph on $n$ vertices. For a vertex $v\in V(G)$ and a
subgraph $H$ of $G$, we use $N_H(v)$ to denote the set, and $d_H(v)$
the number, of neighbors of $v$ in $H$, respectively. We call
$d_H(v)$ the \emph{degree} of $v$ in $H$. When no confusion occurs,
we will denote $N_G(v)$ and $d_G(v)$ by $N(v)$ and $d(v)$,
respectively. The vertex $v$ is called {\em heavy} if $d(v)\geq
n/2$, and a cycle $C$ is called {\em heavy} if $C$ contains all
heavy vertices of $G$.

The following theorem on the existence of heavy cycles in graphs is
well known.

\begin{theorem}[Bollob\'as and Brightwell \cite{Bollobas_Brightwell}, Shi \cite{Shi}]
Every 2-connected graph has a heavy cycle.
\end{theorem}

Let $G$ be a graph, $v$ be a vertex, and $e$ be an edge of $G$. We
use $G-v$ to denote the graph obtained from $G$ by deleting $v$, and
$G-e$ the graph obtained from $G$ by deleting $e$.

Here we first characterize the separable graphs which contain no heavy
cycles.

\begin{theorem}
Let $G$ be a connected graph on $n$ vertices. If $G$ contains no
heavy cycles, then $G$ has at most two heavy vertices. Moreover,\\
(1) if $G$ contains no heavy vertices, then $G$ is a tree;\\
(2) if $G$ contains only one heavy vertex, say $x$, then $G-x$
contains at least $n/2$ components, and each component of $G-x$
contains exactly one neighbor of $x$;\\
(3) if $G$ has exactly two heavy vertices, say $x$ and $y$, then $xy\in
E(G)$ and $xy$ is a cut edge of $G$, $n$ is even and both components
of $G-xy$ have $n/2$ vertices, and $x$ (and $y$, respectively) is
adjacent to every vertices in the component which contains $x$
($y$). Briefly stated, $T_1\subseteq G\subseteq T_2$ (see Fig. 1).
\end{theorem}

\begin{center}
\begin{picture}(290,320)

\thicklines

\put(0,150){

\put(70,90){\put(0,0){\circle*{6}} \put(52,30){\circle*{6}}
\put(30,52){\circle*{6}} \put(0,60){\circle*{6}}
\put(-30,52){\circle*{6}} \put(52,-30){\circle*{6}}
\put(30,-52){\circle*{6}} \put(0,-60){\circle*{6}}
\put(-30,-52){\circle*{6}} \lbezier(0,0)(52,30)
\lbezier(0,0)(52,-30) \lbezier(-30,-52)(30,52)
\lbezier(30,-52)(-30,52) \put(0,-60){\line(0,1){120}}
\lbezier[2](-50,-20)(-50,20) \put(-10,-5){$x$} \put(57,25){$x_1$}
\put(35,47){$x_2$} \put(5,55){$x_3$} \put(-25,47){$x_4$}
\put(47,-40){$x_{n/2-1}$}}

\put(220,90){\put(0,0){\circle*{6}} \put(-52,30){\circle*{6}}
\put(-30,52){\circle*{6}} \put(0,60){\circle*{6}}
\put(30,52){\circle*{6}} \put(-52,-30){\circle*{6}}
\put(-30,-52){\circle*{6}} \put(0,-60){\circle*{6}}
\put(30,-52){\circle*{6}} \lbezier(0,0)(-52,30)
\lbezier(0,0)(-52,-30) \lbezier(-30,-52)(30,52)
\lbezier(30,-52)(-30,52) \put(0,-60){\line(0,1){120}}
\lbezier[2](50,-20)(50,20) \put(5,-5){$y$} \put(-47,30){$y_1$}
\put(-25,47){$y_2$} \put(5,55){$y_3$} \put(35,47){$y_4$}
\put(-57,-40){$y_{n/2-1}$}}

\put(70,90){\line(1,0){150}} \put(140,10){$T_1$}}

\multiput(70,90)(150,0){2}{\put(0,0){\circle*{6}}
{\thinlines\put(0,0){\circle{120}}} \put(-10,30){$K_{n/2}$}}

\put(60,85){$x$} \put(225,85){$y$} \put(70,90){\line(1,0){150}}
\put(140,10){$T_2$}

\end{picture}

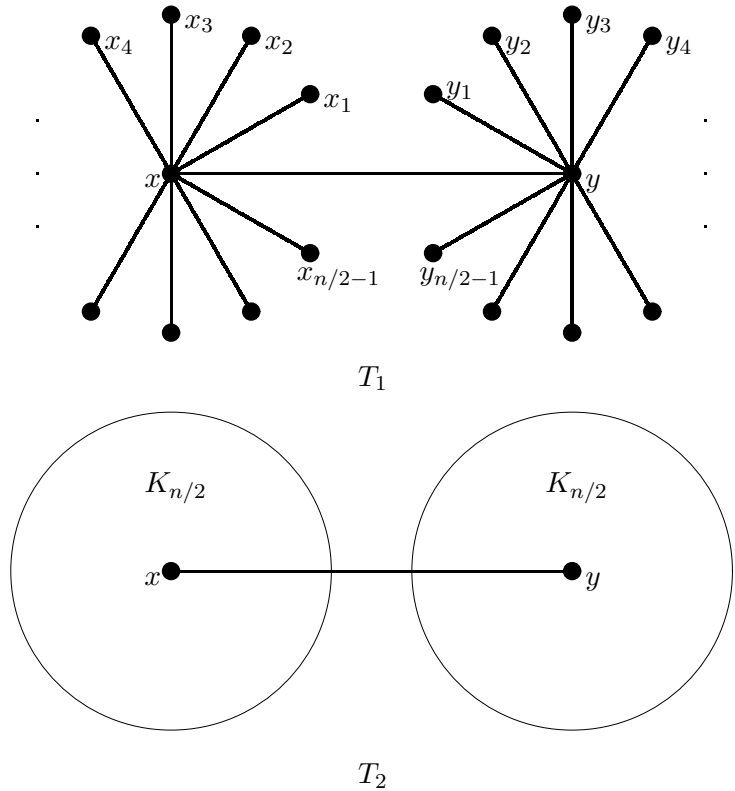
\small Fig. 1. Extremal graphs with two heavy vertices and no heavy
cycles.
\end{center}

We postpone the proof of Theorem 2 in Section 3.

Let $x$ and $y$ be two vertices of a graph $G$, an
$(x,y)$-\emph{path} is a path connecting the two vertices $x$ and
$y$. The \emph{distance} between $x$ and $y$, denote by $d(x,y)$, is
the length of a shortest $(x,y)$-path in $G$.

Let $H$ be a subgraph of a graph $G$. If $H$ contains all edges
$xy\in E(G)$ with $x,y\in V(H)$, then $H$ is called an \emph{induced
subgraph} of $G$. Let $X$ be a subset of $V(G)$. The induced
subgraph of $G$ with vertex set $X$ is called a subgraph
\emph{induced by} $X$, and is denoted by $G[X]$. We use $G-X$ to
denote the subgraph induced by $V(G)\setminus X$, and use the
notation $G-H$ instead of $G-V(H)$.

Let $G$ be a graph on $n$ vertices. For a given graph $H$, we say
that $G$ is $H$-\emph{free} if $G$ does not contain an induced
subgraph isomorphic to $H$. If $H$ is an induced subgraph of $G$, we
say that $H$ is \emph{heavy} in $G$ if there are two nonadjacent
vertices in $V(H)$ with degree sum at least $n$. The graph $G$ is
called $H$-\emph{heavy} if every induced subgraph of $G$ isomorphic
to $H$ is heavy. Note that an $H$-free graph is also $H$-heavy, and
if $H_1$ is an induced subgraph of $H_2$, then an $H_1$-free
($H_1$-heavy) graph is also $H_2$-free ($H_2$-heavy).

In general, a longest cycle of a graph may not be a heavy cycle (see
Fig. 2). In this paper, we mainly consider heavy subgraph conditions
for longest cycles to be heavy. First, consider the following
theorem.

\begin{theorem}[Fan \cite{Fan}]
Let $G$ be a 2-connected graph. If $\max\{d(u),d(v)\}\geq n/2$ for
every pair of vertices with distance 2 in $G$, then $G$ is
Hamiltonian.
\end{theorem}

This theorem implies that every 2-connected $P_3$-heavy graph has a
Hamilton cycle, which is of course a heavy cycle. In fact we have
the following theorem.

\begin{theorem}
If $G$ is a 2-connected $K_{1,4}$-heavy graph, and $C$ is a longest
cycle of $G$, then $C$ is a heavy cycle of $G$.
\end{theorem}

We postpone the proof of this theorem in Section 4.

Note that $K_{1,3}$ is an induced subgraph of $K_{1,4}$. So any longest cycle of
a $K_{1,3}$-heavy graph is heavy. We can get the following result.

\begin{theorem}
Let $S$ be a connected graph on at least 3 vertices and $G$ be a
2-connected graph. Then $G$ being $S$-free (or $S$-heavy) implies
every longest cycle of $G$ is a heavy cycle, if and only if
$S=P_3,K_{1,3}$ or $K_{1,4}$.
\end{theorem}

The sufficiency of this theorem follows from Theorem 4 immediately. We
will prove its necessity in Section 5.

\section{Preliminaries}

Let $C$ be a cycle with a given orientation and $x,y\in V(C)$. We
use $\overrightarrow{C}[x,y]$ to denote the path from $x$ to $y$ on
$\overrightarrow{C}$, and $\overleftarrow{C}[y,x]$ to denote the
path $\overrightarrow{C}[x,y]$ with the opposite direction.

Let $G$ be a graph on $n$ vertices and $k\geq 3$ be an integer. We
call a circular sequence of vertices $C=v_1v_2\cdots v_kv_1$ an
\emph{$o$-cycle} of $G$, if for all $i$ with $1\leq i\leq k$, either
$v_iv_{i+1}\in E(G)$ or $d(v_i)+d(v_{i+1})\geq n$, where
$v_{k+1}=v_1$. The \emph{deficit} of $C$ is defined by
$\de(C)=|\{i: v_iv_{i+1}\notin E(G)$ with $1\leq i\leq k\}|$. Thus a
cycle is an $o$-cycle with deficit 0.

Similarly, we can define \emph{$o$-paths} of $G$.

Now, we prove the following lemma on $o$-cycles.

\begin{lemma}
Let $G$ be a graph and $C$ be an $o$-cycle of $G$. Then there exists
a cycle of $G$ which contains all the vertices in $V(C)$.
\end{lemma}

\begin{proof}
Assume the opposite. Let $C'$ be an $o$-cycle which contains all the
vertices in $V(C)$ such that $\de(C')$ is as small as possible. Then
we have $\de(C')\geq 1$. Without loss of generality, we suppose that
$C'=v_1v_2\cdots v_kv_1$, where $v_1v_k\notin E(G)$ and
$d(v_1)+d(v_k)\geq n$. We use $P$ to denote the $o$-path
$P=v_1v_2\cdots v_k$.

If $v_1$ and $v_k$ have a common neighbor in $V(G)\backslash V(P)$,
denote it by $x$, then $C''=Pv_kxv_1$ is an $o$-cycle which contains
all the vertices in $V(C)$ with deficit smaller than
$\de(C')$, a contradiction.

So we assume that $N_{G-P}(v_1)\cap N_{G-P}(v_k)=\emptyset$. Then we
have $d_P(v_1)+d_P(v_k)\geq |V(P)|$ by $d(v_1)+d(v_k)\geq n$. Thus,
there exists $i$ with $2\leq i\leq k-1$ such that $v_i\in N_P(v_1)$
and $v_{i-1}\in N_P(v_k)$, and then
$C''=P[v_1,v_{i-1}]v_{i-1}v_kP[v_k,v_i]v_iv_1$ is an $o$-cycle which
contains all the vertices in $V(C)$ with deficit smaller than
$\de(C')$, a contradiction.
\end{proof}

Note that Theorem 1 can be induced by Lemma 1.

Let $P$ be an $(x,y)$-path (or $o$-path) of $G$. If the number of
vertices of $P$ is more than that of a longest cycle of $G$, then,
by Lemma 2, we have $xy\notin E(G)$ and $d(x)+d(y)<n$.

In the following, we use $\overline{E}(G)$ to denote the set $\{uv:
uv\in E(G)$ or $d(u)+d(v)\geq n\}$.

\section{Proof of Theorem 2}

If $G$ contains at least three heavy vertices, then let
$X=\{x_1,x_2,\ldots,x_k\}$ be the set of heavy vertices of $G$,
where $k\geq 3$. Thus $C=x_1x_2\cdots x_kx_1$ is an $o$-cycle. By
Lemma 1, there exists a cycle containing all the vertices in $X$,
which is a heavy cycle, a contradiction. Thus we have that $G$
contains at most two heavy vertices.

\begin{case}
$G$ contains no heavy vertices.
\end{case}

If $G$ contains a cycle $C$, then $C$ is a heavy cycle of $G$, a
contradiction. Since $G$ is connected, we have that $G$ is a tree.

\begin{case}
$G$ contains only one heavy vertex.
\end{case}

Let $x$ be the heavy vertex and $H$ be a components of $G-x$. Since
$G$ is connected, we have that $N_H(x)\neq\emptyset$. If
$|N_H(x)|\geq 2$, then let $x_1$ and $x_2$ be two vertices in
$N_H(x)$, and $P$ be an $(x_1,x_2)$-path in $H$. Then $C=Px_2xx_1$
is a cycle containing $x$, which is a heavy cycle, a contradiction.
Thus we have $|N_H(x)|=1$.

Since $d(x)\geq n/2$, we have that $G-x$ contains at least $n/2$
components.

\begin{case}
$G$ contains exactly two heavy vertices.
\end{case}

Let $x$ and $y$ be the two heavy vertices and $P$ be a longest
$(x,y)$-path of $G$. If $|V(P)|\geq 3$, then $C=Pyx$ is an $o$-cycle
of $G$. By Lemma 1, there exists a cycle containing all the vertices
in $V(C)$, which is a heavy cycle, a contradiction. Thus we have
that $|V(P)|=2$, which implies that $xy\in E(G)$ and $xy$ is a cut
edge of $G$.

Let $H_x$ and $H_y$ be the components of $G-xy$ which contains $x$
and $y$, respectively. Since $d(x)\geq n/2$ and $xy'\notin E(G)$ for
all $y'\in V(H_y)\backslash\{y\}$, we have that $|V(H_y)|\leq n/2$.
Similarly we have that $|V(H_x)|\leq n/2$. This implies that $n$ is
even and $|V(H_x)|=|V(H_y)|=n/2$.

By $d(x)\geq n/2$ and $|V(H_x)|=n/2$, we have that $xx'\in E(G)$
for every $x'\in V(H_x)\backslash\{x\}$. Similarly, we have that
$yy'\in E(G)$ for every $y'\in V(H_y)\backslash\{y\}$.

The proof is complete.\hfill$\Box$

\section{Proof of Theorem 4}

We use $n$ to denote the order of $G$, and $c$ the length of $C$. We
give an orientation to $C$. Let $x$ be a vertex in $V(G-C)$, we
prove that $d(x)<n/2$.

Let $H$ be the component of $G-C$ which contains $x$. Then all the
neighbors of $x$ is in $V(C)\cup V(H)$. Let $h=|V(H)|$. Note that
$x$ is not a neighbor of itself, we have $d_H(x)<h$.

\begin{claim}
If $v_1,v_2$ are two vertices in $V(C)$ such that $v_1v_2\in E(C)$,
then either $xv_1\notin E(G)$ or $xv_2\notin E(G)$.
\end{claim}

\begin{proof}
Otherwise, $C-v_1v_2\cup v_1xv_2$ is a cycle longer than $C$, a
contradiction.
\end{proof}

By Claim 1, we have that if $P$ is a subpath of $C$, then
$d_P(x)\leq\lceil |V(P)|/2\rceil$.

By the 2-connectedness of $G$, there exists a $(u_0,v_0)$-path (and
then, a $(u_0,v_0)$-$o$-path) passing though $x$ which is internally
disjoint with $C$, where $u_0,v_0\in V(C)$. We choose an $o$-path
$Q=x_{-k}x_{-k+1}\cdots x_{-1}xx_1\cdots x_l$ such that\\
(1) $x_{\pm 1}\in N(x)$, and\\
(2) $|V(Q)\cap N_H(x)|$ is as large as possible,\\
where $x_{-k}\in V(C)$ and $x_l\in V(C)$.

\begin{claim}
$Q$ contains at least half of the vertices in $N_H(x)$.
\end{claim}

\begin{proof}
If $d_H(x)=0$, then the assertion is obvious. So we assume that
$d_H(x)\geq 1$.

Suppose that $|N_H(x)\cap V(Q)|<d_H(x)/2$. Then $|N_H(x)\backslash
V(Q)|\geq \lceil d_H(x)/2\rceil\geq 1$.

\begin{subclaim}
For every $x'\in N_H(x)\backslash V(Q)$, $x'x_1\notin
\overline{E}(G)$ and $x'x_{-1}\notin \overline{E}(G)$.
\end{subclaim}

\begin{proof}
If $x'x_1\in \overline{E}(G)$, then $Q'=Q[x_{-k},x]xx'x_1Q[x_1,x_l]$
is an $o$-path which contains more vertices in $N_H(x)$ than $Q$, a
contradiction. Thus we have $x'x_1\notin \overline{E}(G)$.

The second assertion can be proved similarly.
\end{proof}

\begin{subclaim}
$x_{-1}x_1\in \overline{E}(G)$.
\end{subclaim}

\begin{proof}
Suppose that $x_{-1}x_1\notin \overline{E}(G)$. Let $x'_i,x'_j$ be
any pair of vertices in $N_H(x)\backslash V(Q)$. By Claim 2.1, we
have that $x'_ix_{\pm 1}\notin \overline{E}(G)$ and $x'_jx_{\pm
1}\notin \overline{E}(G)$. Since $G$ is a $K_{1,4}$-heavy graph, we
have that $x'_ix'_j\in \overline{E}(G)$.

By the 2-connectedness of $G$, there is a path from
$N_H(x)\backslash V(Q)$ to $V(C)\cup V(Q)$ not passing through $x$.
Let $R'=y_1y_2\cdots y_r$ be such a path, where $y_1\in
N_H(x)\backslash V(Q)$ and $y_r\in V(C)\cup V(Q)\backslash\{x\}$.
Let $R$ be an $o$-path from $x$ to $y_1$ passing through all the
vertices in $N_H(x)\backslash V(Q)$.

If $y_r\in V(C)\backslash\{x_{-k},x_l\}$, then $Q'=Q[x_{-k},x]RR'$
is an $o$-path which contains at least half of the vertices in
$N_H(x)$, a contradiction.

If $y_r\in V(Q[x_1,x_l])$, then $Q'=Q[x_{-k},x]RR'Q[y_r,x_l]$ is an
$o$-path which contains at least half of the vertices in $N_H(x)$, a
contradiction.

If $y_r\in V(Q[x_{-k},x_{-1}])$, then we can prove the result
similarly.

Thus the claim holds.
\end{proof}

Now, we choose an $o$-path $R=xx'_1x'_2\cdots x'_r$ which is
internally disjoint with $C\cup Q$, where $x'_r\in V(C)\cup
V(Q)\backslash\{x\}$
such that\\
(1) $x'_1\in N(x)$, and\\
(2) $|V(R)\cap(N_H(x)\backslash V(Q))|$ is as large as possible.

\begin{subclaim}
$R$ contains at least half of the vertices in $N_H(x)\backslash
V(Q)$.
\end{subclaim}

\begin{proof}
Note that $d_{H-Q}(x)\geq 1$. It is easy to know that $x'_1\in
N_H(x)\backslash V(Q)$. By Claim 2.1, we have that
$x'_1x_1\notin\overline{E}(G)$.

Suppose that $|V(R)\cap(N_H(x)\backslash V(Q))|<d_{H-Q}(x)/2$. Let
$N_H(x)\backslash V(Q)\backslash V(R)=\{x''_1,x''_2,\ldots,x''_s\}$,
where $s\geq \lceil d_{H-Q}(x)/2\rceil$.

For every vertex $x''_i\in N_H(x)\backslash V(Q)\backslash V(R)$, by
Claim 2.1, we have that $x''_ix_1\notin\overline{E}(G)$. Similarly,
we can prove that $x''_ix'_1\notin\overline{E}(G)$.

For any pair of vertices $x''_i,x''_j\in N_H(x)\backslash
V(Q)\backslash V(R)$, we have that $x''_ix_1\notin \overline{E}(G)$,
$x''_ix'_1\notin \overline{E}(G)$, $x''_jx_1\notin \overline{E}(G)$,
$x''_jx'_1\notin \overline{E}(G)$ and
$x'_1x_1\notin\overline{E}(G)$. Since $G$ is $K_{1,4}$-heavy, we
have that $x''_ix''_j\in \overline{E}(G)$.

By the 2-connectedness of $G$, there is a path from
$N_H(x)\backslash V(Q)\backslash V(R)$ to $V(C)\cup V(Q)$ not
passing through $x$. Let $T'=y_1y_2\ldots y_t$ be such a path, where
$y_1\in N_H(x)\backslash V(Q)\backslash V(R)$ and $y_t\in V(C)\cup
V(Q)\backslash\{x\}$. Let $T$ be an $o$-path from $x$ to $y_1$
passing through all the vertices in $N_H(x)\backslash V(Q)\backslash
V(R)$. Then $R'=TT'$ is an $o$-path from $x$ to $V(C)\cup
V(Q)\backslash\{x\}$ which contains at least half of the vertices in
$N_H(x)\backslash V(Q)$, a contradiction.
\end{proof}

By Claim 2.3, we have that $R$ contains at least one fourth of the
vertices in $N_H(x)$.

\begin{subclaim}
$x'_r\in V(C)\backslash\{x_{-k},x_l\}$.
\end{subclaim}

\begin{proof}
Assume the opposite. Without loss of generality, we assume that
$x'_r\in[x_1,x_l]$.

If $x'_r=x_1$, then $Q'=Q[x_{-k},x]RQ[x_1,x_l]$ is an $o$-path which
contains more vertices in $N_H(x)$ than $Q$, a contradiction.

If $x'_r=x_i$, where $2\leq i\leq l$, then let $x_j$ be the last
vertex in $[x_1,x_{i-1}]$ such that $x_j\in N(x)$. Then
$Q'=Q[x_{-k},x_{-1}]x_{-1}x_1Q[x_1,x_j]x_jxRQ[x'_r,x_l]$ is an
$o$-path which contains more vertices in $N_H(x)$ than $Q$, a
contradiction.

Thus we have $x'_r\in V(C)\backslash\{x_{-k},x_l\}$.
\end{proof}

If $Q[x,x_l]$ contains fewer than one fourth of the vertices in
$N_H(x)$, then $Q'=Q[x_{-k},x]R$ is an $o$-path which contains more
vertices in $N_H(x)$ than $Q$, a contradiction. This implies that
$Q[x,x_l]$ contains at least one fourth of the vertices in $N_H(x)$.
Similarly, we have that $Q[x_{-k},x]$ contains at least one fourth
of the vertices in $N_H(x)$. Thus $Q$ contains at least half of the
vertices in $N_H(x)$, a contradiction.
\end{proof}

By Claim 2, we have that $k+l-2\geq d_H(x)/2$.

Let $u_0=x_{-k}\in V(C)$ and $v_0=x_l\in V(C)$. We assume that the
length of $\overrightarrow{C}[v_0,u_0]$ is $r_1+1$ and length of
$\overrightarrow{C}[u_0,v_0]$ is $r_2+1$, where $r_1+r_2+2=c$. We
use $\overrightarrow{C}=v_0v_1v_2\cdots
v_{r_1}u_0v_{-r_2}v_{-r_2+1}\cdots v_{-1}v_0$ to denote $C$ with the
given orientation, and $\overleftarrow{C}=u_0u_1u_2\cdots
u_{r_1}v_0u_{-r_2}u_{-r_2+1}$ $\cdots u_{-1}u_0$ to denote $C$ with
the opposite direction, where $v_i=u_{r_1+1-i}$ and
$v_{-j}=u_{-r_2-1+j}$.

\begin{claim}
$r_1\geq k+l-1$, and for every vertex $v_s\in[v_1,v_l]$, $xv_s\notin
E(G)$, and for every vertex $u_t\in[u_1,u_k]$, $xu_t\notin E(G)$.
\end{claim}

\begin{proof}
Note that $Q$ contains $k+l-1$ vertices in $V(H)$. If $r_1<k+l-1$,
then $Pc=Q\overleftarrow{C}[v_0,u_0]$ is an $o$-cycle longer than
$C$. By Lemma 1, there exists a cycle which contains all the
vertices in $V(Pc)$, a contradiction. Thus, we have $r_1\geq k+l-1$.

If $xv_s\in E(G)$, where $v_s\in[v_1,v_l]$, then
$Pc=\overrightarrow{C}[v_s,v_0]Q[v_0,x]xv_s$ is an $o$-cycle which
contains all the vertices in $V(C)\backslash[v_1,v_{s-1}]\cup
V(Q[x,x_{l-1}])$, and $|V(Pc)|>c$, a contradiction.

If $xu_t\in E(G)$, where $u_t\in[u_1,u_k]$, then we can prove the
result similarly.
\end{proof}

Similarly, we can prove the following claim.

\begin{claim}
$r_2\geq k+l-1$, and for every vertex $v_{-s}\in[v_{-l},v_{-1}]$,
$xv_{-s}\notin E(G)$, and for every vertex $u_{-t}\in[u_{-k},u_{-1}]$,
$xu_{-t}\notin E(G)$.
\end{claim}

Let $d_1=d_{\overrightarrow{C}[v_1,u_1]}(x)$ and
$d_2=d_{\overleftarrow{C}[v_{-1},u_{-1}]}(x)$. Then $d_C(x)\leq
d_1+d_2+2$.

\begin{claim}
$d_1\leq (r_1-(k+l)+1)/2$, $d_2\leq (r_2-(k+l)+1)/2$.
\end{claim}

\begin{proof}
If $r_1=k+l-1$, then by Claim 3, we have $d_1=0$. So we assume that
$r_1\geq k+l$.

By Claim 3, we have that
$d_1=d_{\overrightarrow{C}[v_{l+1},u_{k+1}]}(x)$. By Claim 1, we
have that $d_1\leq\lceil(r_1-(k+l))/2\rceil\leq(r_1-(k+l)+1)/2$.

The second assertion can be proved similarly.
\end{proof}

By Claim 5, we have that
$$d_C(x)\leq d_1+d_2+2\leq (r_1+r_2+2-2(k+l))/2+2=c/2-(k+l-2).$$
Note that $k+l-2\geq d_H(x)/2$, we have $d_C(x)\leq(c-d_H(x))/2$.
Thus $d(x)=d_C(x)+d_H(x)\leq(c+d_H(x))/2<(c+h)/2\leq n/2$.

The proof is complete.

\section{Proof of the necessity of Theorem 5}

Note that an $S$-free graph is also $S$-heavy, we only need to prove
that a longest cycle of a 2-connected $S$-free graph is not
necessarily a heavy cycle for $S\neq P_3,K_{1,3}$ and $K_{1,4}$.

First consider the following fact: if a connected graph $S$ on at
least 3 vertices is not $P_3, K_{1,3}$ or $K_{1,4}$, then $S$ must
contain $K_3,P_4,C_4$ or $K_{1,5}$ as an induced subgraph. Thus we
only need to show that not every longest cycle in a $K_3,P_4,C_4$ or
$K_{1,5}$-free graph is heavy.

We construct three graphs $G_1,G_2$ and $G_3$ (see Fig. 2).

\begin{center}
\begin{picture}(240,570)

\thicklines

\put(40,430){\put(80,80){\circle{120}} \put(140,80){\circle*{4}}
\put(20,80){\circle*{4}} \put(138,95.5){\circle*{4}}
\put(138,64.5){\circle*{4}} \put(22,95.5){\circle*{4}}
\put(22,64.5){\circle*{4}} \put(132,110){\circle*{4}}
\put(132,50){\circle*{4}} \put(28,110){\circle*{4}}
\put(28,50){\circle*{4}} \put(122.4,122.4){\circle*{4}}
\put(122.4,37.6){\circle*{4}} \put(145,77.5){$v$} \put(10,77.5){$u$}
\put(143,93){$v_1$} \put(143,62){$v_{-1}$} \put(9.5,93){$v_r$}
\put(4.5,62){$v_{-r}$} \put(137,107.5){$v_2$}
\put(137,47.5){$v_{-2}$}

\put(50,80){\circle*{4}} \put(110,80){\circle*{4}}
\put(80,40){\circle*{4}} \put(80,60){\circle*{4}}
\put(80,100){\circle*{4}} \put(80,120){\circle*{4}}
\put(20,80){\line(1,0){30}} \put(110,80){\line(1,0){30}}
\put(50,80){\line(3,-4){30}} \put(50,80){\line(3,-2){30}}
\put(50,80){\line(3,2){30}} \put(50,80){\line(3,4){30}}
\put(110,80){\line(-3,-4){30}} \put(110,80){\line(-3,-2){30}}
\put(110,80){\line(-3,2){30}} \put(110,80){\line(-3,4){30}}
\lbezier[4](80,60)(80,100) \put(40,82.5){$x$} \put(115,82.5){$y$}
\put(77.5,32.5){$z_1$} \put(77.5,52.5){$z_2$}
\put(71.5,105){$z_{k-1}$} \put(77.5,125){$z_k$}

\put(20,0){$G_1$ ($r\geq 4$ and $k\geq 2r+2$)}}

\put(20,240){\put(100,100){\circle*{4}} \put(117.3,110){\circle*{4}}
\put(117.3,90){\circle*{4}} \put(110,117.3){\circle*{4}}
\put(110,82.7){\circle*{4}} \put(90,117.3){\circle*{4}}
\put(90,82.7){\circle*{4}} \put(100,120){\circle*{4}}
\put(100,80){\circle*{4}} \lbezier(100,100)(117.3,110)
\lbezier(100,100)(117.3,90) \lbezier(90,82.7)(110,117.3)
\lbezier(90,117.3)(110,82.7) \put(100,80){\line(0,1){40}}
\lbezier[3](82.5,90)(82.5,110) \put(90,92.5){$x$}
\put(119.8,107.5){$z_1$} \put(119.8,87.5){$z_k$}
\put(112.5,113.8){$z_2$}

{\thinlines\put(100,40){\circle{40}}} \put(95,35){$K_r$}
{\thinlines\put(100,160){\circle{40}}} \put(95,155){$K_r$}

\put(20,100){\circle*{4}} \put(180,100){\circle*{4}}
{\thinlines\qbezier(20,100)(60,20)(100,20)
\qbezier(20,100)(60,60)(100,60) \qbezier(20,100)(60,80)(100,80)
\qbezier(20,100)(60,120)(100,120) \qbezier(20,100)(60,140)(100,140)
\qbezier(20,100)(60,180)(100,180) \qbezier(180,100)(140,20)(100,20)
\qbezier(180,100)(140,60)(100,60) \qbezier(180,100)(140,80)(100,80)
\qbezier(180,100)(140,120)(100,120)
\qbezier(180,100)(140,140)(100,140)
\qbezier(180,100)(140,180)(100,180)} \qbezier(20,100)(60,70)(100,70)
\qbezier(180,100)(140,70)(100,70) \put(20,100){\line(1,0){160}}
\put(10,97.5){$u$} \put(182.5,97.5){$v$}

\put(40,0){$G_2$ ($r\geq 4$ and $k\geq 2r-1$)}}

\put(0,10){\put(120,120){\circle{200}} \put(220,120){\circle*{4}}
\put(20,120){\circle*{4}} \put(216.6,145.9){\circle*{4}}
\put(216.6,94.1){\circle*{4}} \put(23.4,145.9){\circle*{4}}
\put(23.4,94.1){\circle*{4}} \put(206.6,170){\circle*{4}}
\put(206.6,70){\circle*{4}} \put(33.4,170){\circle*{4}}
\put(33.4,70){\circle*{4}} \put(190.7,190.7){\circle*{4}}
\put(190.7,49.3){\circle*{4}} \put(225,117.5){$v$}
\put(10,117.5){$u$} \put(221.6,143.4){$v_1$}
\put(221.6,91.6){$v_{-1}$} \put(9.5,143.4){$v_r$}
\put(4.5,91.6){$v_{-r}$} \put(211.6,167.5){$v_2$}
\put(211.6,67.5){$v_{-2}$}

\multiput(120,60)(0,60){3}{{\thinlines\put(0,0){\circle{40}}}
\put(-5,-5){$K_k$}}

\put(50,120){\circle*{4}} \put(190,120){\circle*{4}}
{\thinlines\qbezier(50,120)(90,40)(120,40)
\qbezier(50,120)(90,80)(120,80) \qbezier(50,120)(90,100)(120,100)
\qbezier(50,120)(90,140)(120,140) \qbezier(50,120)(90,160)(120,160)
\qbezier(50,120)(90,200)(120,200) \qbezier(190,120)(150,40)(120,40)
\qbezier(190,120)(150,80)(120,80)
\qbezier(190,120)(150,100)(120,100)
\qbezier(190,120)(150,140)(120,140)
\qbezier(190,120)(150,160)(120,160)
\qbezier(190,120)(150,200)(120,200)} \put(20,120){\line(1,0){30}}
\put(190,120){\line(1,0){30}} \put(40,122.5){$x$}
\put(195,122.5){$y$}

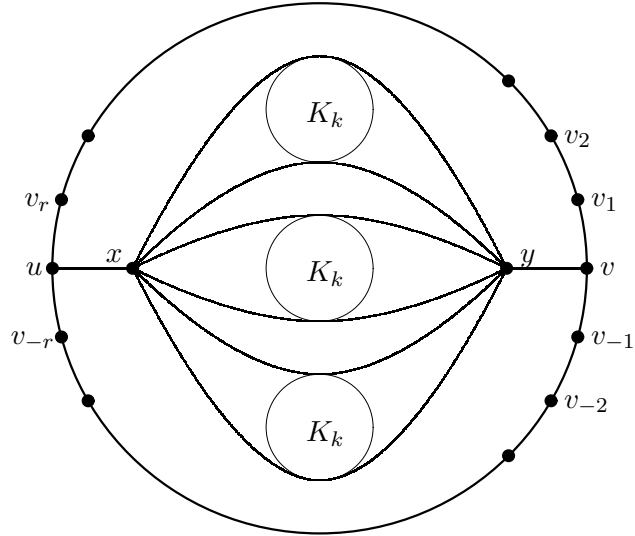
\put(30,0){$G_3$ ($r\geq 11$ and $(2r+2)/3\leq k\leq r-3$)}}

\end{picture}

\small Fig. 2. Graphs $G_1$, $G_2$ and $G_3$.
\end{center}

\begin{remark}
In graph $G_2$, the subgraph $G_2[\{x\}\cup[z_1,z_k]]$ is a star
$K_{1,k}$, and $u$ and $v$ are adjacent to all the vertices in the
$K_{1,k}$ and the two $K_r$'s (note that $uv\in E(G_2)$). In graph
$G_3$, $x$ and $y$ are adjacent to all the vertices in the three
$K_k$'s.

\end{remark}

Note that $G_1$ is $K_3$-free, $G_2$ is $P_4$ and $C_4$-free and
$G_3$ is $K_{1,5}$-free, and the longest cycles of the three graphs
are all not heavy. Thus the necessity of the theorem holds.

\end{document}